\documentclass[12pt, a4paper, reqno, oneside]{amsart}
\usepackage{amssymb}  
\usepackage{amsmath} 
\usepackage{amsthm}  
\usepackage{graphicx}
\usepackage{color}

\usepackage[cp1251]{inputenc}
\usepackage[english]{babel}
\usepackage{cite, verbatim}
\usepackage{multirow}

\newtheorem{theorem}{Theorem}
\newtheorem{lemma}{Lemma}
\theoremstyle{definition}
\newtheorem{remark}{Remark}

\newcommand{\Aut}{\operatorname{Aut}}

\begin{document}


\vspace{1cm}

\title[On primitive prime divisors]{On primitive prime divisors of the~orders of Suzuki--Ree groups (corrected version)}

\author{Maria A. Grechkoseeva}

\thanks{The research was carried out within the framework of the Sobolev Institute of Mathematics state contract (project FWNF-2022-0002).}

\begin{abstract}

There is a well-known factorization of the number $2^{2m}+1$, with $m$ odd, related to the orders of tori of simple Suzuki groups: $2^{2m}+1$ is a product of $a=2^m+2^{(m+1)/2}+1$ and $b=2^m-2^{(m+1)/2}+1$. By the Bang--Zsigmondy theorem, there is a primitive prime divisor of $2^{4m}-1$, that is, a prime $r$ that divides $2^{4m}-1$ and does not divide  $2^i-1$ for any $i<4m$. It is easy to see that $r$ divides $2^{2m}+1$, and so it divides one of the numbers $a$ and $b$. The main objective of this paper is to show that for every  $m>5$, each of $a$ and $b$ is divisible by some primitive prime divisor of $2^{4m}-1$. Also we prove similar results for primitive prime divisors related to the simple Ree groups. As an application, we find the independence and  2-independence numbers of the prime graphs of almost simple Suzuki--Ree groups.

 \end{abstract}

\maketitle

\section*{Introduction}

This is a corrected version of \cite{23Gre.t}. The statement of Theorem 3 of \cite{23Gre.t} concerning independent numbers for almost simple groups with socle $^2G_2(3^m)$ is incorrect. Here we give a corrected version. The difference between \cite{23Gre.t} and the present paper is Item (d) in Theorem \ref{t:independence},  the part of the proof of Theorem \ref{t:independence} in the case $L={}^2G_2(3^m)$ and Remark~3.

A primitive prime divisor of $q^m-1$, where $q$ and $m$ are integers larger than~$1$, is a~prime that divides $q^m-1$ and does not divide $q^i-1$ for all $i<m$.
It is easy to see that a primitive prime divisor of $q^m-1$ always divides $\Phi_m(q)$, where $\Phi_m(x)$ is the $m$th cyclotomic polynomial. Bang \cite{86Bang} proved that primitive prime divisors exist provided that $q$ and $m$ are not too small. Also this fact follows from a later and more general result of Zsigmondy  \cite{Zs}.

\begin{theorem}[{Bang--Zsigmondy}]\label{t:bz} Let $q$ and $m$ be integers larger than $1$.
Then either there is a prime that divides $q^m-1$ and does not divided $q^i-1$ for all $i<m$, or one of the following holds:
\begin{enumerate}
 \item[(a)] $m=2$, $q=2^s-1$, where $s\geqslant 2$;
 \item[(b)] $m=6$, $q=2$.
\end{enumerate}
\end{theorem}

Primitive prime divisors play a prominent role in the theory of finite groups of Lie type,  since
for a finite group of Lie type over a field of size $q$, the properties of its elements of prime order $r$ largely depend not on $r$ itself but on the smallest integer $m$ such that $r$ divides $q^m-1$. In the present paper we are concerned with prime divisors of the orders of the Suzuki--Ree groups, that is, of the groups $^2B_2(2^m)$, $^2G_2(3^m)$ and $^2F_4(2^m)$, where $m$ is an odd prime.

Let $L$ be one of the groups $^2B_2(2^m)$, $^2G_2(3^m)$, or $^2F_4(2^m)$, and we define respectively  $v=2$, $3$, $2$, and $n=4$, $6$, $12$. Then the order of $L$ is divisible by $\Phi_n(v^m)$ but $L$ does not include a torus of order $\Phi_n(v^m)$. Instead, $L$ includes tori of orders $\Psi_n(\sqrt{v^m})$ and $\Psi_n(-\sqrt{v^m})$,
where
$$\Psi_4(x)=x^2+\sqrt{2}x+1;$$
\begin{equation}\label{e:psi}\Psi_6(x)=x^2+\sqrt{3}x+1;\end{equation}
$$\Psi_{12}(x)=x^4+\sqrt{2}x^3+x^2+\sqrt{2}x+1.$$
Note that $\Psi_n(\sqrt{v^m})\Psi_n(-\sqrt{v^m})=\Phi_n(v^m)$ and the numbers $\Psi_n(\sqrt{v^m})$ and $\Psi_n(-\sqrt{v^m})$ are coprime.

Let $m>1$. By the Bang--Zsigmondy theorem, there is a primitive prime divisor $r$ of $v^{mn}-1$. It is clear that $r$ is also a primitive prime divisor of $(v^m)^n-1$, and so $r$ divides $\Phi_n(v^m)$. Thus at least one of the numbers $\Psi_n(\sqrt{v^m})$ and $\Psi_n(-\sqrt{v^m})$ is divisible by a~primitive prime divisor of $v^{mn}-1$. Our main goal is to show that, with two exceptions, each of the numbers $\Psi_n(\sqrt{v^m})$ and $\Psi_n(-\sqrt{v^m})$ is divisible by some primitive prime divisor of $v^{mn}-1$; in particular, $v^{mn}-1$ has at least two primitive prime divisors.

\begin{theorem}\label{t:primitive} Let $v=2,3,2$ and $n=4,6,12$, respectively, let
$\Psi_n(x)$ be the polynomial defined in \eqref{e:psi}, and suppose that $m$ is an odd integer larger than $1$. Then for every $\varepsilon\in\{+,-\}$, either there is a prime that divides $\Psi_{n}(\varepsilon\sqrt{v^{m}})$ and does not divide $v^{i}-1$ for any $i<nm$, or
$v=2$, $n=4$, and $m=3, 5$.
\end{theorem}

Observe that the cases when $v=2$, $n=4$ and $m=3,5$ are indeed exceptional:  $13$ and $41$ are
the only primitive prime divisors of $2^{12}-1$ and $2^{20}-1$ respectively.

If we digress from the Suzuki--Ree groups, we can ask whether there are factorizations of the form $\Phi_n(v^m)=\Psi_n(\sqrt{v^m})\Psi_n(-\sqrt{v^m})$ for other values of $v$ and $n$ and whether the corresponding analog of Theorem \ref{t:primitive} holds. Some comments on this subject are given in Remark 1 at the end of \S\,1. Here we just mention that such factorizations are referred to as Aurifeuillian factorizations (see, for example, \cite[p. 309]{94Rie}) and they can be derived from Lucas's formulas for cyclotomic polynomials  and from other, more general, formulas provided in \cite[Theorem 1]{62Sch}.

In the second part of the paper, we use Theorem \ref{t:primitive} to find some parameters of the prime graph of an almost simple group whose socle is a Suzuki or Ree group. The prime graph (or Gruneberg--Kegel graph) of a finite group $G$ is the graph $GK(G)$ whose vertex set is the set of prime divisors of the order of $G$ and in which two primes $r$ and $s$ are adjacent if and only if $r\neq s$ and $G$ contains an element of order $rs$. A set of vertices that are pairwise not adjacent is said to be independent. We write $t(G)$ to denote the largest size of an independent set of vertices in the graph $GK(G)$ and $t(2,G)$ to denote the largest size of an independent set containing~$2$. Imposing conditions on these numbers, one can substantially restrict the structure of $G$: for example,  \cite[Theorem 1]{57Hig} implies that $G$ is nonsolvable whenever $t(G)\geqslant 3$, and by \cite{05Vas.t}, we have that a group $G$ with $t(2,G)\geqslant 2$ has at most one nonabelian composition factor. Both of these results are of great importance to the problems of recognition of finite groups  by the set of element orders and by the prime graph (see the surveys \cite{22Survey} and \cite{22CamMas} respectively).

The numbers $t(L)$ and $t(2,L)$ for all finite simple groups $L$ are found in \cite{05VasVd.t} (see also \cite{11VasVd.t} for corrections). The next natural step is to find these numbers for all almost simple groups, that is, for all $G$ such that $L\leq G\leq \Aut L$, where $L$ is a nonabelian simple group. The groups $^2B_2(2^m)$, $^2G_2(3^m)$, and $^2F_4(2^m)$ are simple if and only if  $m>1$.

\begin{theorem}\label{t:independence}
Let $L$ be a finite simple Suzuki or Ree group and $L<G\leq\Aut L$. Then one of the following holds:
\begin{enumerate}
\item[(a)] $t(G)=t(L)$ and $t(2,G)=t(2,L)$;
\item[(b)] $L={}^2B_2(32)$, $t(G)=t(L)-1$ and $t(2,G)=t(2,L)-1$;
\item[(c)] $L={}^2F_4(2^m)$, $\pi(G)\not\subseteq\pi(L)$, $t(G)=t(L)+1$, and $t(2,G)=t(2,L)$.
\item[(d)] $L={}^2G_2(3^m)$, $\pi(G)=\pi(L)$, $3$ divides $|G:L|$, $t(G)=t(L)-1$ and $t(2,G)=t(2,L)$.
\end{enumerate}
\end{theorem}

We use the following standard notation. By $\pi(a)$ and $\pi(G)$ we denote the set of prime divisors of an integer $a$ and of the order of a finite group $G$ respectively. We write  $(a,b)$ for the greatest common divisor of integers $a$ and $b$. By $\varphi(x)$ we denote the Euler totient function. If $x$ is a complex number, then $\overline x$ denotes the complex conjugate of $x$.

\section*{\bf \S\,1. Primitive divisors}

Let $m\geqslant 3$, and let $k_m(q)$ be the largest divisor of $q^m-1$ that is coprime to $q^i-1$ for all $i<m$. By \cite[Theorem 3]{Zs} (see also \cite[Proposition 2]{97Roi}),
if  $r$ is the largest prime divisor of $m$ then \begin{equation}\label{e:km}k_m(q)=\frac{\Phi_m(q)}{(\Phi_m(q),r)}.\end{equation}
Moreover, if $r$ divides $\Phi_m(q)$, then $(r,q)=1$ and $r-1$ is divisible by $l$, where $l$ is the $r'$-part of $m$, that is, $m=r^il$ and $(l,r)=1$.

We proceed now to prove Theorem \ref{t:primitive}. We will use, without explicit references, basic properties of cyclotomic polynomials, which can be found, for example, in   \cite[3.3.1--3.3.4]{04Pras}.

It is clear that a prime required in the statement of the theorem exists if and only if \begin{equation}\label{e:in}(k_{nm}(v),\Psi_n(\varepsilon\sqrt{v^m}))>1.\end{equation}
By \eqref{e:km}, the inequality~\eqref{e:in} is a consequence of the inequality
\begin{equation}\label{e:r}(\Phi_{nm}(v), \Psi_n(\varepsilon\sqrt{v^m}))>r,\end{equation}
where $r$ is the largest prime divisor of $nm$, and so of $m$.

\begin{lemma}\label{l:gcd}
Let $(m,n)=1$ and let $s$ be an odd positive integer.
\begin{enumerate}
\item[(a)] If $n=4,6$, then $(\Phi_{nm}(v^s), \Psi_n(\varepsilon\sqrt{v^{sm}}))=|\Phi_m(\sqrt{v^s}/\xi)\Phi_m(\sqrt{v^s}/\overline \xi)|,$ where $\xi$ is some suitable primitive
$2n$th root of unity.
\item[(b)] If $n=12$, then $(\Phi_{nm}(v^s),\Psi_{n}(\varepsilon \sqrt{v^{sm}}))=|\Phi_{3m}(\sqrt{v^s}/\xi)\Phi_{3m}(\sqrt{v^s}/\overline\xi)|,$ where $\xi$ is some suitable primitive $8$th root of unity.
 \end{enumerate}
\end{lemma}

\begin{proof}
(a) Suppose that $n=4,6$. Observe that $\varphi(2n)=4$ and that the primitive $2n$th roots of unity
can be written as $\pm\xi$, $\pm\overline \xi$, where $\xi$ is some fixed root.
Since $2n$ and $m$ are coprime, if $\mu$ runs over the set of primitive $m$th roots of unity, then $\pm\xi\mu$, $\pm\overline\xi\mu$ are all of the primitive $2nm$th roots of unity. So we have
$$\Phi_{2nm}(x)=\Phi_m(x/\xi)\Phi_m(x/\overline\xi)\Phi_m(-x/\xi)\Phi_m(-x/\overline\xi).$$
Furthermore, since $n$ is even, it follows that $\Phi_{nm}(x^2)=\Phi_{2nm}(x)$.
Thus  \begin{equation}\label{e:factor}\Phi_{nm}(x^2)=f_m(x)f_m(-x),\end{equation} where $$f_m(x)=\Phi_m(x/\xi)\Phi_m(x/\overline\xi).$$

It is clear that the coefficients of the polynomial $f_m(x)$ are real. Also it is easy to see that the coefficients of the terms $x^{2i}$ are integral linear combinations of numbers of the form  $\xi^{2j}+\overline\xi^{2j}$, while the coefficients of $x^{2i+1}$ are integral linear combinations of numbers of the form $\xi^{2j+1}+\overline\xi^{2j+1}$. Since  $$\xi^{2j}+\overline\xi^{2j}\in\{0,\pm 2\}\text{ and }\xi^{2j+1}+\overline\xi^{2j+1}=\pm\sqrt{2}\text{ if }n=4,$$ $$\xi^{2j}+\overline\xi^{2j}\in\{\pm 1,\pm 2\}\text{ and }\xi^{2j+1}+\overline\xi^{2j+1}\in\{0,\pm \sqrt{3}\}\text{ if }n=6,$$
we deduce that $$f_m(x)=x^{2\varphi(m)}+a_1\sqrt{v}x^{2\varphi(m)-1}+a_2x^{2\varphi(m)-2}+\dots+1,$$ where $a_1,a_2,\dots$ are integers. Thus $f_m(\sqrt{v^s})$ is an integer for every odd $s$.

Note that  $$\underset{d\mid m}{\prod}f_d(x)=(x^m-\xi^m)(x^m-\overline \xi^m)=x^{2m}-(\xi^m+\overline \xi^m)x^m+1=\Psi_n(\epsilon x^m),$$ where $\epsilon\in\{+,-\}$ is defined by the equality  $\xi^m+\overline\xi^m=-\epsilon\sqrt{v}$. Replacing, if necessary, the root $\xi$, we may assume that $\epsilon=\varepsilon$. By the above, every number $f_d(\sqrt{v^s})$, where $d$ divides $m$, is an integer, and so $f_m(\sqrt{v^s})$ divides $\Psi_n(\varepsilon\sqrt{v^{ms}})$. Similarly,
 $f_m(-\sqrt{v^s})$ divides $\Psi_n(-\varepsilon\sqrt{v^{ms}})$. The numbers $\Psi_n(\sqrt{v^{ms}})$ and $\Psi_n(-\sqrt{v^{ms}})$ are coprime, and hence $(\Phi_{nm}(v^s), \Psi_n(\varepsilon\sqrt{v^{ms}}))=|f_m(\sqrt{v^s})|$, as required.

 (b) Applying \eqref{e:factor} with $n=4$, we derive that $\Phi_{12m}(x^2)=f_{3m}(x)f_{3m}(-x)$.
Also since $3$ does not divide $m$, we have  $$\prod_{d\mid m}\Phi_{3d}(x)=\prod_{d\mid 3m}\Phi_d(x)/\prod_{d\mid m}\Phi_d(x)=\frac{x^{3m}-1}{x^m-1}=\Phi_3(x^m).$$
It follows that
 $$\prod_{d\mid m}f_{3d}(x)=\Phi_3((x/\xi)^m)\Phi_3((x/\overline \xi)^m))=(x^{2m}-\xi^mx^m+\xi^{2m})(x^{2m}-\overline\xi^mx^m+\overline\xi^{2m})=$$
 $$=x^{4m}+\epsilon x^{3m}+x^{2m}+\epsilon x^m+1=\Psi_{12}(\epsilon x^m),$$
 where $\epsilon$ is defined in the same way as in (a). The rest of the proof is similar to the proof of (a).
\end{proof}

\begin{lemma}\label{l:bounds}
Let $n=4,6$, $(m,n)=1$ and $f_m(x)= \Phi_m(x/\xi)\Phi_m(x/\overline\xi)$, where $\xi$ is a~primitive $2n$th root of unity. If either $x\geqslant \sqrt{2}$ and $m>5$, or
$x\geqslant \sqrt{3}$ and $m>3$, then $|f_m(x)|>r$, where $r$ is the largest prime divisor of $m$. \end{lemma}

\begin{proof}
Let $m=r^il$, where $(r,l)=1$. We have $\Phi_m(x)=\Phi_l(x^{r^{i}})/\Phi_l(x^{r^{i-1}})$ and so $$|f_m(x)|=\frac{|\Phi_l(({x}/\xi)^{r^{i}})\Phi_l(({x}/\overline\xi)^{r^{i}})|}{|\Phi_l(({x}/\xi)^{r^{i-1}})\Phi_l(({x}/\overline\xi)^{r^{i-1}})|}.$$
The nominator of the last expression is a product of $\varphi(l)$ factors of the form $|y^{2r}-(\lambda+\overline\lambda)y^r+1|$, where $y=x^{r^{i-1}}$ and $|\lambda|=1$. Since  $y\geqslant \sqrt{2}$, each of these factors is not less than $y^{2r}-2y^r+1=(y^r-1)^2$. Similarly, every factor in the denominator is not larger than $(y+1)^{2}$. Hence
$$|f_m(x)|\geqslant \left(\frac{(y^r-1)^2}{(y+1)^2}\right)^{\varphi(l)}.$$
It ie easy to check that
$$\frac{y^r-1}{y+1}>\frac{y^{r-2}(y^2-1)}{y+1}=y^{r-2}(y-1)\geqslant y^{r-2} \text{ if } y\geqslant 2,$$
$$\frac{y^r-1}{y+1}\geqslant \frac{y^{r-3}(y^3-1)}{y+1}>y^{r-3} \text{ if } y\geqslant \sqrt{3};$$
$$\frac{y^r-1}{y+1}\geqslant \frac{y^{r-5}(y^5-1)}{y+1}>y^{r-5}y^{3/2}=y^{r-\frac{7}{2}} \text{ if } r\geqslant 5.$$

Thus for all $r\geqslant 7$, we have $|f_m({x})|>y^{2(r-\frac{7}{2})}>y^{2(r-4)}\geqslant 2^{r-4}>r.$

If $r=3$, then $i>1$, and so $y\geqslant \sqrt{2^3}>2$, and $|f_m(x)|>y^2\geqslant 2^3>3.$

Let $r=5$. If $x\geqslant \sqrt{3}$ or $i>1$, then $y\geqslant \sqrt{3}$ and hence $|f_m({x})|>y^4\geqslant3^{2}>5.$  If $x<\sqrt{3}$ and $i=1$, then $l\geqslant 3$, and therefore $|f_m(x)|>(y^{3})^{\varphi(l)}\geqslant \sqrt{2}^6>5.$
 \end{proof}

We now show that in all cases, except for the case $n=4$ and $m<5$,  one of \eqref{e:in} and \eqref{e:r} holds.

If $n=4$ and $m>5$, then by Lemmas \ref{l:gcd} and \ref{l:bounds}, we have 
$$(\Phi_{nm}(v),\Psi_n(\varepsilon\sqrt{v^m}))=|f_m(\sqrt{v})|>r.$$

Let $n=6$ and $m=sd$, where $s=3^i$ and $(3,d)=1$. If $d=1$, then $$k_{6m}(v)=\Phi_{6m}(v)=\Phi_6(v^m),$$ where the first equality follows from \eqref{e:km} and the remark after it.
Hence $$(k_{6m}(v),\Psi_6(\varepsilon\sqrt{v^m}))=\Psi_6(\varepsilon\sqrt{v^m})>1.$$
If $d>1$, then $d\geqslant 5$ and $r$ is the largest prime divisor of $d$. Thus, applying Lemmas \ref{l:gcd} and \ref{l:bounds}, we see that $$(\Phi_{6m}(v),\Psi_6(\varepsilon\sqrt{v^m}))=(\Phi_{6d}(v^s), \Psi_6(\varepsilon\sqrt{v^{sd}}))=|f_d(\sqrt{v^s})|>r.$$

The case $n=12$ is similar to the case $n=6$.
This completes the proof of Theorem~\ref{t:primitive}.

\begin{remark}
Let $k\geqslant 2$ be a square-free positive integer and suppose that $v$ is a divisor of $k$ such that $k/v$ is odd. Let $\tau=-$ if $v\equiv 1\pmod 4$ and $\tau=+$ otherwise.  By  \cite[Theorem 1]{62Sch}, there are polynomials $P_{k,v}(x)$, $Q_{k,v}(x)\in \mathbb{Z}[x]$ such that 
$\Phi_{2k}(\tau x)=P_{k,v}^2(x)-vxQ^2_{k,v}(x).$
Hence for every odd positive integer $m$, we have the following factorization of $\Phi_{2k}(\tau v^m)$ as a product of two integers:
\begin{equation}\label{e:auri}
\Phi_{2k}(\tau v^m)=(P_{k,v}(v^m)+v^{(m+1)/2} Q_{k,v}(v^m))(P_{k,v}(v^m)-v^{(m+1)/2} Q_{k,v}(v^m)).\end{equation}
The factorizations of $\Phi_4(2^m)$, $\Phi_6(3^m)$ and $\Phi_{12}(2^m)$ related to the Suzuki--Ree groups are in fact the factorization  \eqref{e:auri} with  $k=2,3, 6$ and $v=2,3,2$ respectively.
Since $\pi(\Phi_{2k}(\tau v^m))$ contains all primitive divisors of $v^{2mk}-1$ (if $\tau=+$) or of $v^{mk}-1$ (if $\tau=-$), one can ask under what conditions on $k$, $v$ and $m$, each of the factors in \eqref{e:auri} is divisible by the corresponding primitive divisor, and thus obtain some generalization of Theorem \ref{t:primitive}. However, the study of this question is beyond the scope of the present work. Also it is worth mentioning that \cite[Theorem 2]{62Sch} provides some sufficient conditions for 
 $v^{2mk}-1$ (or $v^{mk}-1$) to have at least two primitive divisors.
\end{remark}

\begin{remark} Primitive prime divisors of the orders of Suzuki--Ree groups are also 
studied in \cite[Section 3]{22Katz}, and Theorem \ref{t:primitive} of the present paper is Theorems 3.4, 3.7 and Corollary 3.5 in \cite{22Katz}. Observe that the author 
discovered the work \cite{22Katz} after submitting the present paper to ``Algebra and Logic'' and that the proofs 
of these two works, despite some common motifs,  are quite different.
\end{remark}

\section*{\bf \S\,2. Prime graphs}

Let $L$ be one of the groups $^2B_2(2^m)$, $^2G_2(3^m)$, or $^2F_4(2^m)$, where $m>1$ is odd, and let $v=2,3,2$ respectively. It was established in  \cite{62Suz,  61ReeF, 61ReeG} that  the group $\Aut L$ is a semidirect product of $L$ and $\langle \alpha\rangle$, where $\alpha$ is a field automorphism of $L$ of order $m$. Furthermore, every outer automorphism of $L$ of prime order is conjugate to a generator of $\langle \alpha^k\rangle$ for some $k$ dividing $m$ and $C_L(\alpha^k)$ is a Suzuki--Ree group of the same type as $L$ but over the field of order $v^{k}$ (see, for example, \cite[Proposition 4.9.1]{98GorLySol}).

Note the following consequences of these facts valid for every $G$ such that $L\leq G\leq \Aut L$.  Firstly, 
$\Aut L/L$ is cyclic, and so $t(G)\leqslant t(L)+1$.
Secondly, $\alpha$ centralizes an element of order $2$, and hence $t(2,G)\leqslant t(2,L)$. Thirdly, 
$|G/L|$ cannot be divisible by a primitive prime divisor $r$ of $v^{mi}-1$ for any $i\geqslant 1$. Indeed,  $mi$ is equal to the multiplicative order of $v$ modulo  $r$ and 
thus $mi$ divides $r-1$, which implies that $r>|\Aut L/L|$.

We prove Theorem \ref{t:independence} considering each type of the simple groups in turn. Information about the numbers $t(L)$, $t(2,L)$ and independent sets whose sizes are equal to these numbers are taken from \cite[Table 4]{11VasVd.t} and \cite[Propositions  6.5 and 6.8]{05VasVd.t} respectively.

Suppose that $L={}^2B_2(2^m)$. Then $t(L)=t(2,L)=4$. Also the independent sets of vertices of size four in $GK(L)$ are exactly the sets of the form $\{2,r_1,r_4^+, r_4^-\}$,
where $r_1\in\pi(2^m-1)$ and $r_4^\varepsilon\in\pi(\Psi_4(\varepsilon \sqrt{2^m}))$. Assume that $\rho$ is an independent set of size five in $GK(G)$. Then $\rho\cap\pi(G/L)=\{r\}$ for  some $r$ and $|\rho\cap\pi(L)|=4$. The last equality implies that $2\in \rho$ and hence $t(2,G)\geqslant 5$, a contradiction.
Thus $t(2,G)\leqslant t(G)\leqslant 4$.

Let $m>5$. By Theorems \ref{t:bz} and \ref{t:primitive}, we can take a primitive prime divisor of $2^m-1$ as $r_1$ 
and primitive prime divisors of $2^{4m}-1$ dividing $\Psi_4(\sqrt{2^m})$ and $\Psi_4(-\sqrt{2^m})$ as 
$r_4^+$ and $r_4^-$ respectively. Then $\{2,r_1,r_4^+, r_4^-\}\cap\pi(G/L)=\varnothing$ and so the set
$\{2,r_1,r_4^+, r_4^-\}$ remains independent in $GK(G)$. Thus $t(2,G)=t(G)=4$, as required. 

If $m=3, 5$, then $GK(L)$ consists of four isolated vertices $2$, $5$, $2^m-1$ and $\Psi_4^+(\sqrt{2^m})$.
If $m=3$, then all of these numbers are coprime to $|G/L|$ and, reasoning as above, we have $t(2,G)=t(G)=4$. If  $m=5$, then $\pi(G)=\pi(L)$ and  $GK(G)$ differs from  $GK(L)$ in having an edge between $5$ and $2$. Thus $t(2,G)=t(G)=3$.

Let $L={}^2G_2(3^m)$, where $m\geq 3$. Then $t(L)=5$ and $t(2,L)=3$, and the independent sets of size five and independent sets of size three containing 2 are exactly the sets of the form   $\{3,r_1, r_2, r_6^+, r_6^-\}$ and
$\{2, r_6^+, r_6^-\}$ respectively, where $2\neq r_1\in\pi(3^m-1)$, $r_2\in\pi(3^m+1)$ and $r_6^\varepsilon\in\pi(\Psi_6(\varepsilon\sqrt{3^m})$.
By Theorem \ref{t:bz}, we can take $r_1$ and $r_2$ to be primitive divisors of $3^m-1$ and $3^{2m}-1$. By Theorem \ref{t:primitive}, we can choose $r_{6}^+$ and $r_{6}^-$ in such a way that they both are primitive prime divisors of  $3^{6m}-1$. All these chosen numbers do not divide $|G/L|$, and hence $t(G)\geqslant 4$ and $t(2,G)\geqslant 3$. It follows that $t(2,G)=3$. Also every prime  divisor of $|G/L|$ is adjacent to $3$ in $GK(G)$ and, therefore, $t(G)\leqslant 5$.

Suppose that $\pi(G)=\pi(L)$. If $3$ does not divide $|G:L|$, then $\{3,r_1, r_2, r_6^+, r_6^-\}$ is still a coclique in $GK(G)$ and $t(G)=5$. Let $3$ divide $|G:L|$ and assume that $\rho$ is a coclique of size $5$ in $GK(G)$. Since $3\cdot\omega(^2G_2(3^{m/3}))\subseteq \omega(G)$, it follows that $3r\in\omega(G)$ for every $r\in\pi(3^m+1)$. Hence $\rho$ cannot contain $3$ and a divisor of $3^m+1$ at the same time, and so $\rho$ is not a coclique in $GK(L)$. This contradiction shows that $t(G)=4$. 

Suppose that $\pi(G)\neq \pi(L)$ and $r\in\pi(G)\setminus\pi(L)$. It follows that $r>3$. Assume that $r$ is adjacent to some of the chosen primitive divisors, the divisor $r_i$ of $3^{mi}-1$ say. Then $r_i\in\pi({}^2G_2(3^{m/r}))$ and so  $r_i$ divides $(3^{mi}-1, 3^{6m/r}-1)=3^{(mi, 6m/r)}-1$. Since $r>3$, we have  $(mi,6m/r)=(ri, 6)m/r=(i,6)m/r<im$, contrary to the definition of primitive divisor. Thus $\{r, r_1, r_2, r_{6}^+, r_{6}^-\}$ is an independent set in $GK(G)$ and $t(G)=5$. 

Suppose that $L={}^2F_4(2^m)$, where $m>3$. Then $t(L)=5$ and $t(2,L)=4$. The independent sets of size five and independent sets of size four containing 2 are exactly the sets of the form $\{r_2, r_4, r_6, r_{12}^+, r_{12}^-\}$ and $\{2, r_6, r_{12}^+, r_{12}^-\}$ respectively, where $3\neq r_2\in\pi(2^m+1)$, $r_4\in\pi(2^{2m}+1)$, $3\neq r_6\in\pi(2^{2m}-2^m+1)$ and $r_{12}^\varepsilon\in\pi(\Psi_{12}(\varepsilon \sqrt{2^m}))$.
We choose $r_i$, with $i=2,4,6$, to be a primitive prime divisor of $2^{mi}-1$ (it exists because $m>3$) and $r_{12}^+$ and $r_{12}^-$ to be primitive prime divisors of  $2^{12m}-1$.
All these chosen numbers do not divide $|G/L|$, and hence $t(G)\geqslant 5$ and $t(2,G)\geqslant 4$. It follows that $t(2,G)=4$. If $\pi(G)\subseteq \pi(L)$, then $t(G)\leqslant t(L)$, and so $t(G)=t(L)$.

Suppose that  $r\in\pi(G)\setminus\pi(L)$. In particular, $r>3$. Assume that $r$ is adjacent to some chosen primitive divisor $r_i$ of $2^{mi}-1$.
Then $r_i\in\pi({}^2F_4(2^{m/r}))$ and so  $r_i$ divides $(2^{mi}-1, 2^{12m/r}-1)=2^{(mi,12m/r)}-1$. Since $r>3$, we have  $(mi,12m/r)=(ri, 12)m/r=im/r<im$, contrary to the definition of primitive divisor. Thus $\{r, r_2, r_4, r_6, r_{12}^+, r_{12}^-\}$ is an independent set in $GK(G)$ and $t(G)=6$.

Let $L={}^2F_4(8)$. In this case $t(L)=t(2,L)=4$ and $\{2, 19, 37, 109\}$ is an independent set. This set is disjoint from $\pi(G/L)$. Also $\pi(G)\subseteq\pi(L)$, and therefore $t(G)=t(L)$ and $t(2,G)=t(2,L)$. The proof of Theorem \ref{t:independence} is complete. 

\begin{remark}
Theorem 3 of \cite{23Gre.t} was used in the proof of Theorem 1 of \cite{24GrePan.t} to show that the spectrum of an almost simple group $G$ with socle $^2G_2(3^m)$ cannot be equal to the spectrum of a classical simple group $H$ \cite[p. 1082]{24GrePan.t}. A corrected argument is as follows. If $\omega(G)=\omega(H)$, then $4\not\in\omega(H)$ and so $H=L_2(u)$ for some $u$ by Walter's classification of simple finite groups with abelian Sylow 2-subgroups \cite{69Wal}. Since $t(L_2(u))=3$ and $t(G)\geqslant 4$ by the corrected version of Theorem 3, we derive a contradiction.
\end{remark}

The author is grateful to A.A. Buturlakin and A.V. Vasil'ev for discussion of the problem and comments to the paper.

\makeatletter
\renewcommand*{\@biblabel}[1]{\hfill#1.}
\makeatother

\medskip

Grechkoseeva Maria

Sobolev Institute of Mathematics

\verb"grechkoseeva@gmail.com"

\end{document}